\documentclass[12pt]{amsart}
\usepackage{amsmath,xcolor}
\allowdisplaybreaks[4]
\usepackage{geometry,amsthm,graphics,tabularx,amssymb,shapepar}
\usepackage{amscd}
\usepackage{bm,bbm}
\usepackage{mathrsfs}
\usepackage{enumerate}
\usepackage[all]{xypic}
\usepackage{BOONDOX-cal}

\newtheorem{thm}{Theorem}[section]
\newtheorem{lem}[thm]{Lemma}
\newtheorem{defi}[thm]{Definition}
\newtheorem{prop}[thm]{Proposition}
\newtheorem{rem}[thm]{Remark}

\allowdisplaybreaks

\numberwithin{equation}{section}
\title[Representations of the affine ageing algebra $\widehat{\mathfrak{age}}(1)$]{Representations of the affine ageing algebra $\widehat{\mathfrak{age}}(1)$}

 \author{Huaimin Li}\address{School of Mathematical Sciences, Xiamen University,
 Xiamen, China 361005} \email{lihuaimin1999@qq.com}
  \author{Qing Wang$^1$}\address{School of Mathematical Sciences, Xiamen University,
 Xiamen, China 361005} \email{qingwang@xmu.edu.cn }\thanks{$^1$Supported by
China NSF grants No.12071385 and No.12161141001.}


\begin{document}

\begin{abstract}
In this paper, we investigate the affine ageing algebra $\widehat{\mathfrak{age}}(1)$,
which is a central extension of the loop algebra of the 1-spatial ageing algebra $\mathfrak{age}(1)$.
Certain Verma-type modules
including Verma modules and imaginary Verma modules of $\widehat{\mathfrak{age}}(1)$ are studied.
Particularly, the simplicity of these modules are characterized and their irreducible quotient modules are determined.
We also study the restricted modules of $\widehat{\mathfrak{age}}(1)$
which are also the modules of the affine vertex algebra arising from the 1-spatial ageing algebra $\mathfrak{age}(1)$.
We present certain constructions of simple restricted $\widehat{\mathfrak{age}}(1)$-modules and an explicit such example of simple
restricted module via the Whittaker module of $\widehat{\mathfrak{age}}(1)$ is given.
\end{abstract}

\maketitle
\section{Introduction}
In physics, the ageing algebra is a local dynamical symmetry of many ageing systems,
far from equilibrium, and with a dynamical exponent $z=2$ \cite{H1,PH}.
From the view of mathematics,
the ageing algebra is a common subalgebra of the corresponding Schr\"odinger algebra \cite{K} and conformal Galilei algebra \cite{HP}
which both are non-semisimple Lie algebras appeared in physics.
Representations of Schr\"odinger algebras and conformal Galilei algebras have been extensively studied rencently in \cite{CHS,CY,GG,LLW,LMZ2} etc.
Thus the study of representation theory of ageing algebras (see \cite{BL1,HS,LMZ1,SH} etc.)
is no doubt helpful for the study of  Schr\"odinger algebras and conformal Galilei algebras.
In particular, L\"u, Mazorchuk and Zhao \cite{LMZ1} gave a complete classification of simple weight modules for the 1-spatial ageing algebra $\mathfrak{age}(1)$
and constructed many new simple weight modules for $(1+1)$-dimensional space-time Schr\"odinger algebra via the simple weight modules of $\mathfrak{age}(1)$.
In this paper, we consider the affine ageing algebra $\widehat{\mathfrak{age}}(1)$,
which is a one-dimensional central extension of the loop algebra of the 1-spatial ageing algebra $\mathfrak{age}(1)$.
We study certain representations of $\widehat{\mathfrak{age}}(1)$.

Verma-type modules are fundamental and important modules in the representation theory of Lie algebra,
they induced from one-dimensional modules for Borel subalgebra associated to certain triangular decomposition.
Verma-type modules for certain Lie algebras have been studied by many authors, see \cite{BBFK,F1,F2,GK,JK1,JK2} etc.
Imaginary Verma modules are important Verma-type modules,
they have both finite and infinite dimensional weight subspaces,
the imaginary Verma modules of affine Kac-Moody algebras have been studied in \cite{F1}.
In this paper, we consider both Verma modules and imaginary Verma modules of $\widehat{\mathfrak{age}}(1)$.
For the Verma modules, we prove that they are not simple and
we determine their irreducible quotients.
For the imaginary Verma modules,
we characterize their simplicity and for those non-simple imaginary Verma modules we determine their irreducible quotients.

The category of restricted modules (also known as smooth modules) of the Lie algebra is important
module category corresponding to the module category of the vertex algebra arising from this Lie algebra.
It is well known that the category of restricted modules
for affine Lie algebras is equivalent to the
category of modules for the corresponding affine vertex algebras \cite{FZ,L}.
In \cite{MZ}, Mazorchuk and Zhao proposed a general construction for simple Virasoro modules which are locally finite over a positive part.
This construction later are extensively used in the study of restricted modules for other Lie algebras related to Virasoro algebra,
such as \cite{CDH,CG,CHS,CY,GG,GL,GX,LPXZ,LS} etc.
For the study of simple restricted $\widehat{\mathfrak{age}}(1)$-modules,
we consider a subalgebra
\begin{equation*}
\widehat{\mathfrak{age}}(1)_{(d_1,d_2,d_3)}=\sum_{i\in\mathbb{N}}(\mathbb{C}h_i\oplus\mathbb{C}e_{d_1+i}\oplus\mathbb{C}p_{d_2+i}\oplus\mathbb{C}q_{d_3+i})\oplus\sum_{i\in\mathbb{Z}}\mathbb{C}z_i\oplus\mathbb{C}k,
\end{equation*}
where $d_1,d_2,d_3\in \mathbb{Z}$. Note that since it is a subalgebra, we have $d_1+d_3\geq d_2$.
For any simple $\widehat{\mathfrak{age}}(1)_{(d_1,d_2,d_3)}$-module $V$, we prove that
under certain conditions the induced module
$\mbox{Ind}_{\widehat{\mathfrak{age}}(1)_{(d_1,d_2,d_3)}}^{\widehat{\mathfrak{age}}(1)}(V)$
is a simple restricted $\widehat{\mathfrak{age}}(1)$-module.
Also, we show that the simple restricted $\widehat{\mathfrak{age}}(1)$-modules
satisfying certain conditions are isomorphic to the induced modules we constructed.
Finally we give an explicit example of this kind of simple restricted $\widehat{\mathfrak{age}}(1)$-module
via the standard Whittaker module of $\widehat{\mathfrak{age}}(1)$.

This paper is organized as follows.
In section \ref{sec:2}, we first introduce and study the affine ageing algebra $\widehat{\mathfrak{age}}(1)$,
then we study the second cohomology group of the loop algebra $\widetilde{\mathfrak{age}}(1)$.
In section \ref{sec:3}, we study the Verma modules and imaginary Verma modules of $\widehat{\mathfrak{age}}(1)$ respectively.
We determine the irreducible quotient modules for those modules which are not simple.
In section \ref{sec:4}, we present a construction
of simple restricted $\widehat{\mathfrak{age}}(1)$-modules, and also give an explicit example via standard Whittaker module.

Throughout the paper,
 $\mathbb{Z}$, $\mathbb{N}$, $\mathbb{Z}_+$ and $\mathbb{C}$ are the sets of integers,
nonnegative integers, positive integers and complex numbers, respectively.
The degree $\mbox{deg}(0)$ is not defined and whenever we write $\mbox{deg}(v)$
we mean $v\ne 0$.

\section{Affine ageing algebra $\widehat{\mathfrak{age}}(1)$}

\label{sec:2}
	\def\theequation{2.\arabic{equation}}
	\setcounter{equation}{0}

In this section, we consider the affine ageing algebra $\widehat{\mathfrak{age}}(1)$,
which is a one-dimensional central extension of the loop algebra of the 1-spatial ageing algebra $\mathfrak{age}(1)$.
We determine the second cohomology group of the loop algebra $\widetilde{\mathfrak{age}}(1)$.
We also recall some notions and introduce some total orders for later use.

The {\em 1-spatial ageing algebra} $\mathfrak{age}(1) $ is a complex Lie algebra spanned by elements $\{e,h,p,q,z\}$
with Lie brackets
\begin{equation}\nonumber
\begin{aligned}
&[h,e]=2e,~[h,p]=p,~[h,q]=-q,\\
&[e,p]=0,~[e,q]=p,~[p,q]=z,\\
&[z,\mathfrak{age}(1)]=0.
\end{aligned}
\end{equation}

The 1-spatial ageing algebra $\mathfrak{age}(1)$ is equipped with the triangular decomposition:
\begin{equation}
\mathfrak{age}(1)=\mathfrak{age}(1)^+\oplus\mathfrak{age}(1)^0\oplus\mathfrak{age}(1)^-,
\end{equation}
where
$ \mathfrak{age}(1)^+=\mathbb{C}e\oplus\mathbb{C}p, \mathfrak{age}(1)^0=\mathbb{C}h\oplus\mathbb{C}z, \mathfrak{age}(1)^-=\mathbb{C}q.$

We consider the one-dimensional central extension of the loop algebra $$\widetilde{\mathfrak{age}}(1)=\mathfrak{age}(1)\otimes\mathbb{C}[t,t^{-1}].$$
The second cohomology group $H^2(\widetilde{\mathfrak{age}}(1),\mathbb{C})$ is defined by
\begin{equation}
H^2(\widetilde{\mathfrak{age}}(1),\mathbb{C}):=Z^2(\widetilde{\mathfrak{age}}(1),\mathbb{C})\big/ B^2(\widetilde{\mathfrak{age}}(1),\mathbb{C}),
\end{equation}
where
\begin{equation}\nonumber
\begin{aligned}
&Z^2(\widetilde{\mathfrak{age}}(1),\mathbb{C}):=\left\{f:\widetilde{\mathfrak{age}}(1)\times\widetilde{\mathfrak{age}}(1)\rightarrow\mathbb{C}\Bigg|
\begin{aligned}
&(\mbox{i}) f(\alpha,[\beta,\gamma])+f(\gamma,[\alpha,\beta])+f(\beta,[\gamma,\alpha])=0,\\
&(\mbox{ii}) f(\alpha,\beta)=-f(\beta,\alpha), \forall \,\,\alpha,\beta,\gamma\in\widetilde{\mathfrak{age}}(1).
\end{aligned}\right\},\\
&B^2(\widetilde{\mathfrak{age}}(1),\mathbb{C}):=\left\{f:\widetilde{\mathfrak{age}}(1)\times\widetilde{\mathfrak{age}}(1)\rightarrow\mathbb{C}\Bigg|
\begin{aligned}
&\exists \,\,g:\widetilde{\mathfrak{age}}(1)\rightarrow\mathbb{C} (\mbox{linear}),\\
&\forall \,\,\alpha,\beta\in\widetilde{\mathfrak{age}}(1),f(\alpha,\beta)=g([\alpha,\beta]).
\end{aligned}\right\}.
\end{aligned}
\end{equation}
It is well known that there exists a one-to-one correspondence between $H^2(\widetilde{\mathfrak{age}}(1),\mathbb{C})$ and the set of equivalence classes of
one-dimensional central extension of $\widetilde{\mathfrak{age}}(1)$.

\begin{lem}
For all $f\in Z^2(\widetilde{\mathfrak{age}}(1),\mathbb{C})$, there exists $f^\prime\in Z^2(\widetilde{\mathfrak{age}}(1),\mathbb{C})$ such that $f^\prime(\alpha\otimes t^m,\beta\otimes t^n)=0$ for all $\alpha,\beta\in \mathfrak{age}(1)$ except the case $\alpha=\beta=h$, $m,n\in\mathbb{Z}$ and $\overline{f^\prime}=\overline{f}$ in $H^2(\widetilde{\mathfrak{age}}(1),\mathbb{C})$.
\end{lem}
\begin{proof}
For any $m,n,l\in\mathbb{Z}$ and $f\in Z^2(\widetilde{\mathfrak{age}}(1),\mathbb{C})$, we have
\begin{equation*}
\begin{aligned}
0&=f(e\otimes t^m,[h\otimes1,e\otimes t^n])+f(h\otimes1,[e\otimes t^n,e\otimes t^m])+f(e\otimes t^n,[e\otimes t^m, h\otimes 1])\\
&=2f(e\otimes t^m,e\otimes t^n)-2f(e\otimes t^n,e\otimes t^m)=4f(e\otimes t^m,e\otimes t^n),
\end{aligned}
\end{equation*}
\begin{equation*}
\begin{aligned}
0&=f(e\otimes t^m,[h\otimes1,p\otimes t^n])+f(h\otimes1,[p\otimes t^n,e\otimes t^m])+f(p\otimes t^n,[e\otimes t^m, h\otimes 1])\\
&=f(e\otimes t^m,p\otimes t^n)-2f(p\otimes t^n,e\otimes t^m)=3f(e\otimes t^m,p\otimes t^n),
\end{aligned}
\end{equation*}
\begin{equation*}
\begin{aligned}
0&=f(q\otimes t^m,[e\otimes t^n,q\otimes 1])+f(e\otimes t^n,[q\otimes 1,q\otimes t^m])+f(q\otimes 1,[q\otimes t^m, e\otimes t^n])\\
&=f(q\otimes t^m,p\otimes t^{n})-f(q\otimes 1,p\otimes t^{m+n}).
\end{aligned}
\end{equation*}
Similarly we have
\begin{equation*}
\begin{aligned}
0&=f(z\otimes t^m,e\otimes t^n)=f(p\otimes t^m,p\otimes t^n)=f(z\otimes t^m,p\otimes t^n)\\
&=f(q\otimes t^m,q\otimes t^n)=f(z\otimes t^m,q\otimes t^n)=f(z\otimes t^m,z\otimes t^n)
\end{aligned}
\end{equation*}
and
\begin{equation*}
\begin{aligned}
&f(h\otimes1,p\otimes t^{m+n})=f(h\otimes t^m,p\otimes t^n)=f(e\otimes t^m,q\otimes t^n),\\
&f(h\otimes1,q\otimes t^{m+n})=f(h\otimes t^m,q\otimes t^n),\\
&f(h\otimes1,e\otimes t^{m+n})=f(h\otimes t^m,e\otimes t^n).
\end{aligned}
\end{equation*}

Then we also have
\begin{equation*}
\begin{aligned}
0&=f(h\otimes t^m,[p\otimes t^n,q\otimes 1])+f(p\otimes t^n,[q\otimes 1,h\otimes t^m])+f(q\otimes 1,[h\otimes t^m, p\otimes t^n])\\
&=f(h\otimes t^m,z\otimes t^n)+f(p\otimes t^n,q\otimes t^m)+f(q\otimes 1,p\otimes t^{m+n})=f(h\otimes t^m,z\otimes t^n).
\end{aligned}
\end{equation*}

Define $f^*:\widetilde{\mathfrak{age}}(1)\times\widetilde{\mathfrak{age}}(1)\rightarrow\mathbb{C}$ with $f^*(h\otimes t^m, h\otimes t^n)=0$ for $m,n\in\mathbb{Z}$ and $f^*=f$ for others.
It is clear that $f^*\in Z^2(\widetilde{\mathfrak{age}}(1),\mathbb{C})$.
For any $m\in\mathbb{Z}$, we define
\begin{equation*}
\begin{aligned}
g:&\widetilde{\mathfrak{age}}(1)\rightarrow\mathbb{C}\\
&h\otimes t^m\longmapsto 0\\
&e\otimes t^m\longmapsto \frac{1}{2}f^\prime(h\otimes1,e\otimes t^{m})\\
&p\otimes t^m\longmapsto f^\prime(h\otimes1,p\otimes t^{m})\\
&q\otimes t^m\longmapsto -f^\prime(h\otimes1,q\otimes t^{m})\\
&z\otimes t^m\longmapsto -f^\prime(q\otimes1,p\otimes t^{m}).
\end{aligned}
\end{equation*}
Then $f^*(\alpha,\beta)=g([\alpha,\beta])$ for any $\alpha,\beta\in\widetilde{\mathfrak{age}}(1)$ and hence $f^*\in B^2(\widetilde{\mathfrak{age}}(1),\mathbb{C})$.
Let $f^\prime=f-f^*$, then $\overline{f^\prime}=\overline{f}$ in $H^2(\widetilde{\mathfrak{age}}(1),\mathbb{C})$ and $f^\prime(\alpha\otimes t^m,\beta\otimes t^n)=0$ for all $\alpha,\beta\in \mathfrak{age}(1)$ except the case $\alpha=\beta=h$, $m,n\in\mathbb{Z}$.
\end{proof}

Then we consider the Lie algebra $\widehat{\mathfrak{age}}(1)=\mathfrak{age}(1)\otimes\mathbb{C}[t,t^{-1}]\oplus\mathbb{C}k$ with Lie brackets
$$[\alpha\otimes t^m, \beta\otimes t^n]=[\alpha,\beta]\otimes t^{m+n}+m(\alpha\mid \beta)\delta_{m+n,0}k,\,\,\,\, [\widehat{\mathfrak{age}}(1), k]=0,$$
for $\alpha,\beta\in\mathfrak{age}(1),m,n\in\mathbb{Z}$,
$(h\mid h)=1$ and others are zero. Then $\widehat{\mathfrak{age}}(1)$ is a one-dimensional central extension of $\widetilde{\mathfrak{age}}(1)$,
we call it {\em affine ageing algebra associated to $\mathfrak{age}(1)$}.
It is clear that $k,z\otimes t^m (m\in\mathbb{Z})$ are central elements.

\begin{rem}
{\em Let $(\cdot\mid\cdot)$ be a invariant symmetric bilinear form on $\mathfrak{age}(1)$, then we have
\begin{equation*}
\begin{aligned}
&(p\mid q)=([e,q]\mid q)=(e\mid[q,q])=0,\\
&(p\mid h)=-([p,h]\mid h)=-(p\mid[h,h])=0,\\
&(p\mid e)=([h,p]\mid e)=(h\mid[p,e])=0.
\end{aligned}
\end{equation*}
Similarly we have $(\alpha\mid \beta)=0$ for all $\alpha,\beta\in\mathfrak{age}(1)$ except the case $\alpha=\beta=h$.}
\end{rem}

It is clear that $\widehat{\mathfrak{age}}(1)$ is $\mathbb{Z}$-graded:
\begin{equation}
\widehat{\mathfrak{age}}(1)=\coprod_{n\in\mathbb{Z}}\widehat{\mathfrak{age}}(1)^{(n)},
\end{equation}
where
$\widehat{\mathfrak{age}}(1)^{(0)}=\mathfrak{age}(1)\oplus\mathbb{C}k, ~~\widehat{\mathfrak{age}}(1)^{(n)}=\mathfrak{age}(1)\otimes t^n~~\mbox{for}~~n\ne 0.$

For convenience, we shall denote $\alpha\otimes t^n$ by $\alpha_n$ for $\alpha\in\mathfrak{age}(1),n\in\mathbb{Z}$.

Finally we recall and introduce some notions and total orders for later use.
Let $\mathbb{M}$ be the set of all infinite vectors of the form
${\bf i}=(...,i_2,i_1)$ with entries in $\mathbb{N}$
such that the number of nonzero entries is finite.
Let {\bf 0} denote the element $(...,0,0)\in \mathbb{M}$ and
$\epsilon_i$ denote the element $(...,0,1,0,...,0)\in \mathbb{M}$ for $i\in \mathbb{Z}_+$,
 where 1 is in the $i$-th position from right.
 For any ${\bf i}\in \mathbb{M}$, we denote
${\bf w(i)}=\sum_{s\in \mathbb{Z}_+}{s\cdot i_s}$ and ${\bf d(i)}=\sum_{s\in \mathbb{Z}_+}{i_s}$,
which are nonnegative integers.
Let $\succ$ be the
 {\em reverse lexicographical} total order on $\mathbb{M}$,
 that is, for any ${\bf i},{\bf j} \in \mathbb{M}$,
\begin{equation*}
{\bf j}\succ {\bf i}\ \Leftrightarrow \ {\rm there\ exists}\ r\in \mathbb{Z}_+\ {\rm such\ that}\
 (j_s=i_s, \forall\; 1\leqslant s<r)\ {\rm and}\  j_r>i_r.
\end{equation*}

\section{Verma-type modules of $\widehat{\mathfrak{age}}(1)$}

\label{sec:3}
	\def\theequation{3.\arabic{equation}}
	\setcounter{equation}{0}

In this section, we study the Verma modules and imaginary Verma modules
of $\widehat{\mathfrak{age}}(1)$.
For the Verma modules, we prove that they are not simple and determine their irreducible quotients.
For the imaginary Verma modules,
we prove that these modules are simple under certain conditions and determine the irreducible quotient for the non-simple modules.

\subsection{Verma modules of $\widehat{\mathfrak{age}}(1)$}

We note that $\widehat{\mathfrak{age}}(1)$ can be equipped with the triangular decomposition:
\begin{equation}
\widehat{\mathfrak{age}}(1)=\widehat{\mathfrak{age}}(1)^+\oplus\widehat{\mathfrak{age}}(1)^0\oplus\widehat{\mathfrak{age}}(1)^-=\widehat{\mathfrak{age}}(1)^{\geq0}\oplus\widehat{\mathfrak{age}}(1)^-,
\end{equation}
where
\begin{equation*}
\begin{aligned}
&\widehat{\mathfrak{age}}(1)^+=\coprod_{n\in \mathbb{Z}_+}\widehat{\mathfrak{age}}(1)^{(n)}, \\
&\widehat{\mathfrak{age}}(1)^0=\widehat{\mathfrak{age}}(1)^{(0)},\\
&\widehat{\mathfrak{age}}(1)^-=\coprod_{-n\in\mathbb{Z}_+}\widehat{\mathfrak{age}}(1)^{(n)},\\
&\widehat{\mathfrak{age}}(1)^{\geq0}=\widehat{\mathfrak{age}}(1)^+\oplus\widehat{\mathfrak{age}}(1)^0.
\end{aligned}
\end{equation*}

For $\mathcal{h},\mathcal{k}\in\mathbb{C}$, we define the Verma module of $\widehat{\mathfrak{age}}(1)$ by
$$M_{V}(\mathcal{h},\mathcal{k})=U(\widehat{\mathfrak{age}}(1))\otimes_{U(\widehat{\mathfrak{age}}(1)^{\geq0})}\mathbb{C}v,$$
where $\mathbb{C}v$ is the one dimensional $\widehat{\mathfrak{age}}(1)^{\geq0}$-module on which $h$ acts as $\mathcal{h}$, $k$ acts as $\mathcal{k}$
and $\widehat{\mathfrak{age}}(1)^+$ acts as zero. Since $\mathbb{C}v$ is an one-dimensional $\widehat{\mathfrak{age}}(1)^{0}$-module,
we have $ev=pv=qv=zv=0$.

For ${\bf i,j,l,m,n}\in \mathbb{M}$, denote
$$h^{\bf i}p^{\bf j}e^{\bf l}q^{\bf m}z^{\bf n}=\cdots h_{-2}^{i_2}h_{-1}^{i_1}\cdots p_{-2}^{j_2}p_{-1}^{j_1}\cdots e_{-2}^{l_2}e_{-1}^{l_1}\cdots q_{-2}^{m_2}q_{-1}^{m_1}
\cdots z_{-2}^{n_2}z_{-1}^{n_1}\in U(\widehat{\mathfrak{age}}(1)).$$
By the PBW Theorem, each element of $M_{V}(\mathcal{h},\mathcal{k})$ can be uniquely written in the form
\begin{equation}\label{eq:idelt0}
  \sum_{{\bf i,j,l,m,n}\in \mathbb{M}}a_{\bf i,j,l,m,n}h^{\bf i}p^{\bf j}e^{\bf l}q^{\bf m}z^{\bf n}v,
\end{equation}
where all $a_{\bf i,j,l,m,n}\in\mathbb{C}$ and only finitely many of them are nonzero.

It is clear that $M_{V}(\mathcal{h},\mathcal{k})$ is $\mathbb{N}$-graded:
\begin{equation}
M_{V}(\mathcal{h},\mathcal{k})=\coprod_{s\in\mathbb{N}}M_{V}(\mathcal{h},\mathcal{k})^{(s)},
\end{equation}
where $M_{V}(\mathcal{h},\mathcal{k})^{(s)}=\bigoplus_{{\bf w(i+j+l+m+n)}=s}\mathbb{C}h^{\bf i}p^{\bf j}e^{\bf l}q^{\bf m}z^{\bf n}v$.

\begin{lem}
For $\mathcal{h},\mathcal{k}\in\mathbb{C}$, the Verma module $M_{V}(\mathcal{h},\mathcal{k})$ is not simple.
\end{lem}
\begin{proof}
Let $v^\prime=z_{-1}v$, we have $\widehat{\mathfrak{age}}(1)^+v^\prime=0$.
Then the submodule of $M_{V}(\mathcal{h},\mathcal{k})$ generated by $v^\prime$ doesn't contain $v$, hence it is a proper submodule.
Therefore $M_{V}(\mathcal{h},\mathcal{k})$ is not simple.
\end{proof}

Now we want to determine the irreducible quotient of the Verma module $M_{V}(\mathcal{h},\mathcal{k})$.
First we give the definition of the primitive vector in a quotient of $M_{V}(\mathcal{h},\mathcal{k})$.

\begin{defi}
{\em For $\mathcal{h},\mathcal{k}\in\mathbb{C}$, let $V(\mathcal{h},\mathcal{k})$ be a quotient of $M_{V}(\mathcal{h},\mathcal{k})$.
For $n\in\mathbb{N}$, a vector $v^\prime\in V(\mathcal{h},\mathcal{k})^{(n)}$ is called {\em primitive} if there exists a submodule $U$ in $V$ such that
$v^\prime\notin U$ and $\widehat{\mathfrak{age}}(1)^+v^\prime\subset U$.}
\end{defi}

\begin{prop}
For $\mathcal{h},\mathcal{k}\in\mathbb{C}$, let $V(\mathcal{h},\mathcal{k})$ be a quotient of $M_{V}(\mathcal{h},\mathcal{k})$.
$V(\mathcal{h},\mathcal{k})$ is simple if and only if there are no primitive vectors in $V(\mathcal{h},\mathcal{k})$ except $\lambda v(\lambda\in\mathbb{C})$.
\end{prop}
\begin{proof}
$\Rightarrow$:
If $v^\prime$ is a primitive vector in $V(\mathcal{h},\mathcal{k})^{(n)}$ for $n\in\mathbb{N}$,
since $V(\mathcal{h},\mathcal{k})$ is simple,
then $\widehat{\mathfrak{age}}(1)^+v^\prime=0$ and $V(\mathcal{h},\mathcal{k})=U(\widehat{\mathfrak{age}}(1)^{\leq0})v^\prime$,
hence $n$ must be 0, then $v^\prime$ is proportional to $v$.

$\Leftarrow$:If $V(\mathcal{h},\mathcal{k})$ is not simple,
then there exists a proper submodule $V^\prime$ and $v\notin V^\prime$.
We take the smallest integer $n$ such that ${V^\prime}^{(n)}\ne0$, let $v^\prime\in {V^\prime}^{(n)}$, we have $\widehat{\mathfrak{age}}(1)^+v^\prime=0$,
then $v^\prime$ is a primitive vectors and $v^\prime$ is not proportional to $v$,
it is a contradiction. Therefore $V(\mathcal{h},\mathcal{k})$ is simple.
\end{proof}

Then we get the main result of this subsection.

\begin{thm}\label{vermathm}
For $\mathcal{h},\mathcal{k}\in\mathbb{C}$,\\
{\rm (1)} if $\mathcal{k}=0$, the irreducible quotient of $M_{V}(\mathcal{h},\mathcal{k})$ is isomorphic to the one-dimensional $\widehat{\mathfrak{age}}(1)$-module $\mathbb{C}w$;\\
{\rm (2)} if $\mathcal{k}\ne0$, the irreducible quotient of $M_{V}(\mathcal{h},\mathcal{k})$ is isomorphic to the $\widehat{\mathfrak{age}}(1)$-module $U_h$ with generator $w$ and relations $\alpha_iw=0$ $(\alpha=e,p,q,z$ and $i\in\mathbb{Z})$, $h_iw=0$ $(i\in\mathbb{Z}_+)$, $hw=\mathcal{h}w, kw=\mathcal{k}w$.
\end{thm}
\begin{proof}
It is clear that $z_iv$ is primitive for all $i\in-\mathbb{Z}_+$.

\vspace{3mm}
\noindent{\bf Claim 1:} $\alpha_iv$ is primitive for all $\alpha=p,q,e$ and $i\in-\mathbb{Z}_+$.
\vspace{3mm}

We prove this claim by induction.
For $i=-1$, since for all $j\in \mathbb{Z}_+$, we have
\begin{equation*}
\begin{aligned}
&h_je_{-1}v=[h_j,e_{-1}]v=2e_{j-1}v=0,\\
&q_je_{-1}v=[q_j,e_{-1}]v=-p_{j-1}v=0,\\
&e_je_{-1}v=e_{-1}e_jv=0, \,\,p_je_{-1}v=e_{-1}p_jv=0 ,\\
&z_je_{-1}v=e_{-1}z_jv=0 ,\\
\end{aligned}
\end{equation*}
then $\widehat{\mathfrak{age}}(1)^+e_{-1}v=0$, hence $e_{-1}v$ is primitive.
Similarly we have $p_{-1}v$ and $q_{-1}v$ is primitive.

Suppose that $\alpha_iv$ is primitive for all $\alpha=p,q,e$ and $-r\leq i\in-\mathbb{Z}_+$.
Let $V^\prime$ be the submodule of $M_{V}(\mathcal{h},\mathcal{k})$ generated by all
$\alpha_iv$ ($\alpha=p,q,e$ and $-r\leq i\in-\mathbb{Z}_+$) and $z_iv$ ($i\in-\mathbb{Z}_+$).
For $i=-r-1$, since for all $j\in \mathbb{Z}_+$, we have
\begin{equation*}
\begin{aligned}
&h_je_{-r-1}v=[h_j,e_{-r-1}]v=2e_{-r-1+j}v\in V^\prime,\\
&q_je_{-r-1}v=[q_j,e_{-r-1}]v=-p_{-r-1+j}v\in V^\prime,\\
&e_je_{-r-1}v=e_{-r-1}e_jv=0,\,\,p_je_{-r-1}v=e_{-r-1}p_jv=0 ,\\
&z_je_{-r-1}v=e_{-r-1}z_jv=0 ,\\
\end{aligned}
\end{equation*}
then $\widehat{\mathfrak{age}}(1)^+e_{-r-1}v\in V^\prime$.
It is clear that $e_{-r-1}v\notin V^\prime$,
 hence $e_{-r-1}v$ is primitive.
Similarly we have $p_{-r-1}v$ and $q_{-r-1}v$ is primitive.

Then we have $\alpha_iv$ is primitive for all $\alpha=p,q,e$ and $i\in-\mathbb{Z}_+$.

\vspace{3mm}
(1)
For $\mathcal{k}=0$, we can prove that $h_iv$ is primitive for all $i\in-\mathbb{Z}_+$ similarly.
Let $V^{\prime\prime}$ be the submodule of $M_{V}(\mathcal{h},\mathcal{k})$ generated by all
$\alpha_iv$ ($\alpha=p,q,e,h,z$ and $i\in-\mathbb{Z}_+$).
Then $M_{V}(\mathcal{h},\mathcal{k})/V^{\prime\prime}$ is one-dimensional, hence it is simple.
Therefore the irreducible quotient of $M_{V}(\mathcal{h},\mathcal{k})$ is isomorphic to the one-dimensional $\widehat{\mathfrak{age}}(1)$-module $\mathbb{C}w$.

\vspace{3mm}
(2)
Let $V^{\prime\prime\prime}$ be the submodule of $M_{V}(\mathcal{h},\mathcal{k})$ generated by all
$\alpha_iv$ ($\alpha=p,q,e,z$ and $i\in-\mathbb{Z}_+$).
Let $U=M_{V}(\mathcal{h},\mathcal{k})/V^{\prime\prime\prime}$,
 each element of $U$ can be uniquely written in the form
\begin{equation}\label{eq}
  \sum_{{\bf i}\in \mathbb{M}}a_{\bf i}h^{\bf i}v+V^{\prime\prime\prime},
\end{equation}
where all $a_{\bf i}\in \mathbb{C}$ and only finitely many of them are nonzero.

For any $v^\prime\in U$ written in the form of (\ref{eq}),
 we denote by $\mbox{supp}(v^\prime)$ the set of all ${\bf i}\in \mathbb{M}$
 such that $a_{\bf i}\ne 0$.
 For a nonzero $v^\prime\in U$,
  let $\mbox{deg}(v^\prime)$ denote the maximal
   (with respect to the reverse lexicographical total order on $\mathbb{M}$)
   element in $\mbox{supp}(v^\prime)$.

   \vspace{3mm}
\noindent{\bf Claim 2:} If $\mbox{deg}(v^\prime)={\bf i}$, let $r={\rm min}\{s:i_s \ne 0\}>0$,
then $\mbox{deg}(h_rv^\prime)={\bf i}-\epsilon_r$, where $\epsilon_r$ is the element $(...,0,1,0,...,0)\in \mathbb{M}$,
 where 1 is in the $r$-th position from right.

\vspace{3mm}
It suffices to consider those ${\bf i^\prime}\in\mbox{supp}(v^\prime)$ with $h_{r}(a_{\bf i^\prime}h^{\bf i^\prime}v+V^{\prime\prime\prime})\ne0+V^{\prime\prime\prime}$,
then $i_{r}^\prime\ne0$.
Since we have
 \begin{equation*}
\begin{aligned}
h_{r}(a_{\bf i^\prime}h^{\bf i^\prime}v+V^{\prime\prime\prime})&=a_{\bf i^\prime}[h_{r},h^{\bf i^\prime}]v+V^{\prime\prime\prime}=a_{\bf i^\prime}h^{{\bf i^\prime}-\epsilon_r}[h_{r},h_{-r}]v+V^{\prime\prime\prime}\\&=
a_{\bf i^\prime}h^{{\bf i^\prime}-\epsilon_r}rkv+V^{\prime\prime\prime}=r\mathcal{k}a_{\bf i^\prime}h^{{\bf i^\prime}-\epsilon_r}v+V^{\prime\prime\prime},
\end{aligned}
\end{equation*}
then $\mbox{deg}(h_{r}(a_{\bf i^\prime}h^{\bf i^\prime}v+V^{\prime\prime\prime}))={\bf i^\prime}-\epsilon_r\preceq {\bf i}-\epsilon_r$,
the equality holds if and only if ${\bf i^\prime=i}$.
Therefore $\mbox{deg}(h_rv^\prime)={\bf i}-\epsilon_r$.

\vspace{3mm}
Using Claim 2 repeatedly,
from any $0\neq v^\prime\in U$
we can get $v+V^{\prime\prime\prime}$,
which gives the simplicity of $U$.
Define
\begin{equation*}
\begin{aligned}
\varphi:&U\rightarrow U_h\\
&\sum_{{\bf i}\in \mathbb{M}}a_{\bf i}h^{\bf i}v+V^{\prime\prime\prime}\longmapsto \sum_{{\bf i}\in \mathbb{M}}a_{\bf i}h^{\bf i}w,
\end{aligned}
\end{equation*}
then $\varphi$ is a $\widehat{\mathfrak{age}}(1)$-module epimorphism, and $\varphi$ is injective by $U$ is simple, hence it is an isomorphism.
\end{proof}

\subsection{Imaginary Verma modules of $\widehat{\mathfrak{age}}(1)$}

We note that $\widehat{\mathfrak{age}}(1)$ can be equipped with another triangular decomposition:
\begin{equation}
\widehat{\mathfrak{age}}(1)=\widehat{\mathfrak{age}}(1)^{[+]}\oplus\widehat{\mathfrak{age}}(1)^{[0]}\oplus\widehat{\mathfrak{age}}(1)^{[-]}=\widehat{\mathfrak{age}}(1)^{[\geq0]}\oplus\widehat{\mathfrak{age}}(1)^{[-]},
\end{equation}
where
 \begin{equation*}
\begin{aligned}
&\widehat{\mathfrak{age}}(1)^{[+]}=\coprod_{n\in \mathbb{Z}}(\mathbb{C}e_n\oplus \mathbb{C}p_n)\oplus\coprod_{n\in \mathbb{Z}_+}(\mathbb{C}h_n), \\
&\widehat{\mathfrak{age}}(1)^{[0]}=\coprod_{n\in \mathbb{Z}}\mathbb{C}z_n\oplus \mathbb{C}h\oplus\mathbb{C}k, \\
&\widehat{\mathfrak{age}}(1)^{[-]}=\coprod_{n\in \mathbb{Z}}\mathbb{C}q_n\oplus\coprod_{-n\in\mathbb{Z}_+}(\mathbb{C}h_n),\\
&\widehat{\mathfrak{age}}(1)^{[\geq0]}=\widehat{\mathfrak{age}}(1)^{[+]}\oplus\widehat{\mathfrak{age}}(1)^{[0]}.
\end{aligned}
\end{equation*}

For $\mathcal{h},\mathcal{k},\mathcal{z}_i(i\in \mathbb{Z})\in\mathbb{C}$, let $(\mathcal{z})=(\cdots,\mathcal{z}_{-1},\mathcal{z}_0,\mathcal{z}_1,\cdots)$,
we define the imaginary Verma module of $\widehat{\mathfrak{age}}(1)$ by
$$M_{IV}(\mathcal{h},\mathcal{k},(\mathcal{z}))=U(\widehat{\mathfrak{age}}(1))\otimes_{U(\widehat{\mathfrak{age}}(1)^{[\geq0]})}\mathbb{C}v,$$
where $\mathbb{C}v$ is the one dimensional $\widehat{\mathfrak{age}}(1)^{[\geq0]}$-module on which $h$ acts as $\mathcal{h}$, $k$ acts as $\mathcal{k}$, $z_i$ acts as $\mathcal{z}_i$ for $i\in\mathbb{Z}$
and $\widehat{\mathfrak{age}}(1)^{[+]}$ acts as zero. Let $(0)=(\cdots,0,0,0,\cdots)$.

For ${\bf i,j}\in \mathbb{M}$ and $d\in\mathbb{Z}$, denote
$$q_d^{\bf i}h^{\bf j}=\cdots q_{d-2}^{i_2}q_{d-1}^{i_1}\cdots h_{-2}^{j_2}h_{-1}^{j_1}\in U(\widehat{\mathfrak{age}}(1)).$$
By the PBW Theorem, for any element $v$ of $M_{IV}(\mathcal{h},\mathcal{k},(\mathcal{z}))$, there exists $d\in\mathbb{Z}$ such that $v$ can be uniquely written in the form
\begin{equation}\label{eq:idelt1}
  \sum_{{\bf i,j}\in \mathbb{M}}a_{\bf i,j}q_d^{\bf i}h^{\bf j}v,
\end{equation}
where all $a_{\bf i,j}\in\mathbb{C}$ and only finitely many of them are nonzero.
For any $v^\prime\in M_{IV}(\mathcal{h},\mathcal{k},(\mathcal{z}))$
  written in the form of (\ref{eq:idelt1}) for some $d\in \mathbb{Z}$,
  we have $\forall \,j\in\mathbb{Z}$, $p_jv^\prime$ can also be uniquely written in the form of (\ref{eq:idelt1}) for $d$.

\begin{defi}
{\em Define a {\em principal total order} on $\mathbb{M}\times\mathbb{M}$,
 still denoted by $\succ$: $\forall\ {\bf i,j,i^\prime,j^\prime} \in \mathbb{M} $,
 we say that $({\bf i,j}) \succ ({\bf i^\prime,j^\prime})$ if one of the following conditions holds.
 \\(i) ${\bf w(i)>w(i^\prime)}$.
 \\(ii) ${\bf w(i)=w(i^\prime),i\succ i^\prime}$.
 \\(iii) ${\bf i=i^\prime,j\succ j^\prime}$.}
\end{defi}

For any $v^\prime\in M_{IV}(\mathcal{h},\mathcal{k},(\mathcal{z}))$ written in the form of (\ref{eq:idelt1}) for some $d\in \mathbb{Z}$,
 we denote by $\mbox{supp}_d(v^\prime)$ the set of all $({\bf i,j)}\in \mathbb{M}\times\mathbb{M}$
 such that $a_{\bf i,j}\ne 0$.
 For a nonzero $v^\prime\in M_{IV}(\mathcal{h},\mathcal{k},(\mathcal{z}))$,
  let $\mbox{deg}_d(v^\prime)$ denote the maximal
   (with respect to the principal total order on $\mathbb{M}\times\mathbb{M}$)
   element in $\mbox{supp}_d(v^\prime)$.

\begin{lem}\label{ivmlem}
For any ${\bf i}\in \mathbb{M}, j\in\mathbb{Z}$, we have $p_jh^{\bf i}v=e_jh^{\bf i}v=0$.
\end{lem}
\begin{proof}
We prove $p_jh^{\bf i}v=0$ for any ${\bf i}\in \mathbb{M}, j\in\mathbb{Z}$ by induction on ${\bf d(i)}$.

For ${\bf d(i)}=0$, i.e., ${\bf i}={\bf 0}$, then for all $j\in\mathbb{Z}$, we have $p_jh^{\bf i}v=p_jv=0$.
For ${\bf d(i)}=1$, $h^{\bf i}v=h_sv$ for some $s\in-\mathbb{Z}_+$, then for all $j\in\mathbb{Z}$, we have $p_jh_sv=[p_j,h_s]v=-p_{j+s}v=0$.

Suppose $p_jh^{\bf i}v=0$ for all ${\bf i}\in \mathbb{M}$ with ${\bf d(i)}<r$ ($r\in\mathbb{Z}_+$) and $j\in\mathbb{Z}$.
Then for ${\bf d(i)}=r$, $j\in\mathbb{Z}$, we have
$$p_jh^{\bf i}v=[p_j,h^{\bf i}]v=\sum h^{\bf i^*}[p_j,h_{-s}]h^{{\bf i-i^*}-\epsilon_s}v=-\sum h^{\bf i^*}p_{j-s}h^{{\bf i-i^*}-\epsilon_s}v=0, $$
where ${\bf i^*}\in\mathbb{M}$ such that $i_c^*=i_c$ for $c>s$ and $i_c^*=0$ for $c\leq s$.

Similarly, for any ${\bf i}\in \mathbb{M}, j\in\mathbb{Z}$, $e_jh^{\bf i}v=0$.
\end{proof}

Then we get the main result of this subsection.

\begin{thm}\label{ivmthm}
For $\mathcal{h},\mathcal{k},\mathcal{z}_i(i\in \mathbb{Z})\in\mathbb{C}$. Suppose $\mathcal{k}\ne0$.\\
{\rm (1)} If there exists $r\in\mathbb{Z}$ such that $\mathcal{z}_r\ne0,\mathcal{z}_{r+i}=0\,(i\in\mathbb{Z}_+)$,
then $M_{IV}(\mathcal{h},\mathcal{k},(\mathcal{z}))$ is simple.\\
{\rm (2)} If there exists $r\in\mathbb{Z}$ such that $\mathcal{z}_r\ne0,\mathcal{z}_{r-i}=0\,(i\in\mathbb{Z}_+)$,
then $M_{IV}(\mathcal{h},\mathcal{k},(\mathcal{z}))$ is simple.
\end{thm}
\begin{proof}
(1) For any $v^\prime\in M_{IV}(\mathcal{h},\mathcal{k},(\mathcal{z}))$ written in the form of (\ref{eq:idelt1}) for some $d\in \mathbb{Z}$,
denote $\mbox{deg}_d(v^\prime)=({\bf i,j}).$

\vspace{3mm}
\noindent{\bf Claim 1:} If ${\bf i}\ne{\bf 0}$, let $a={\rm min}\{s:i_s \ne 0\}>0$, then $\mbox{deg}_d(p_{a-d+r}v^\prime)=({\bf i}-\epsilon_a,{\bf j})$.
\vspace{0.1mm}

It suffices to consider those $({\bf i^\prime,j^\prime})\in \mbox{supp}_d(v^\prime)$ such that $a_{\bf i^\prime,j^\prime}p_{a-d+r}q_t^{\bf i^\prime}h^{\bf j^\prime}v\ne0$.
By Lemma \ref{ivmlem}, we have
$p_{a-d+r}q_d^{\bf i^\prime}h^{\bf j^\prime}v=[p_{a-d+r},q_d^{\bf i^\prime}]h^{\bf j^\prime}v.$

Consider
$q_d^{\bf i^*}[p_{a-d+r},q_{d-s}]q_d^{{\bf i^\prime-i^*}-\epsilon_s}h^{\bf j^\prime}v=q_d^{{\bf i^\prime}-\epsilon_s}h^{\bf j^\prime}z_{a-s+r}v,$
where ${\bf i^*}\in\mathbb{M}$ such that $i_c^*=i_c$ for $c>s$ and $i_c^*=0$ for $c\leq s$.
Since $\mathcal{z}_r\ne0,\mathcal{z}_{r+i}=0\,(i\in\mathbb{Z}_+)$,
then for $s<a$, $q_d^{{\bf i^\prime}-\epsilon_s}h^{\bf j^\prime}z_{a-s+r}v$ is zero; for $s\geq a$,
$$\mbox{deg}_d(q_d^{{\bf i^\prime}-\epsilon_s}h^{\bf j^\prime}z_{a-s+r}v)=({\bf i^\prime}-\epsilon_s,{\bf j^\prime}).$$
If ${\bf w(i)>w(i^\prime)}$ or $s>a$, then
${\bf w(i^\prime}-\epsilon_s)={\bf w(i^\prime)}-s<{\bf w(i)}-a={\bf w(i}-\epsilon_a),$
hence $({\bf i^\prime}-\epsilon_s,{\bf j^\prime})\prec ({\bf i}-\epsilon_a,{\bf j})$.
If ${\bf w(i)=w(i^\prime)}$ and $s=a$, then ${\bf i^\prime\preceq i}$,
hence $({\bf i^\prime}-\epsilon_a,{\bf j^\prime})\preceq ({\bf i}-\epsilon_a,{\bf j})$,
the equality holds if and only if ${\bf i^\prime=i, j^\prime=j}$.

Therefore $\mbox{deg}_d(p_{a-d+r}v^\prime)=({\bf i}-\epsilon_a,{\bf j})$.

\vspace{3mm}
\noindent{\bf Claim 2:} If ${\bf i=0},{\bf j}\ne{\bf 0}$, let $b={\rm min}\{s:j_s \ne 0\}>0$, then $\mbox{deg}_d(h_{b}v^\prime)=({\bf 0},{\bf j}-\epsilon_b)$.
\vspace{3mm}

It suffices to consider those $({\bf 0,j^\prime})\in \mbox{supp}_d(v^\prime)$ such that $a_{\bf 0,j^\prime}h_{b}h^{\bf j^\prime}v\ne0$,
then $j_{b}^\prime\ne0$.
Since we have
 \begin{equation*}
\begin{aligned}
h_{b}h^{\bf j^\prime}v=[h_{b},h^{\bf j^\prime}]v=h^{{\bf j^\prime}-\epsilon_b}[h_{b},h_{-b}]v=
h^{{\bf j^\prime}-\epsilon_b}bkv=b\mathcal{k}h^{{\bf j^\prime}-\epsilon_b}v,
\end{aligned}
\end{equation*}
then $\mbox{deg}_d(h_{b}h^{\bf j^\prime}v)=({\bf 0,j^\prime}-\epsilon_b)\preceq ({\bf 0,j}-\epsilon_b)$,
the equality holds if and only if ${\bf j^\prime=j}$.
Therefore $\mbox{deg}_d(h_{b}v^\prime)=({\bf 0},{\bf j}-\epsilon_b)$.

\vspace{3mm}
From any $0\neq v^\prime\in M_{IV}(\mathcal{h},\mathcal{k},(\mathcal{z}))$,
use Claim 1 and Claim 2 repeatedly,
we can get $v$,
which gives the simplicity of $M_{IV}(\mathcal{h},\mathcal{k},(\mathcal{z}))$.

\vspace{3mm}
(2) We change the meaning of ``$q_d^{\bf i}h^{\bf j}$''.
 For ${\bf i,j}\in \mathbb{M}$ and $d\in\mathbb{Z}$, denote
$$q_d^{\bf i}h^{\bf j}=\cdots q_{d+2}^{i_2}q_{d+1}^{i_1}\cdots h_{-2}^{j_2}h_{-1}^{j_1}\in U(\widehat{\mathfrak{age}}(1)).$$
Then the proof of (2) is similar to (1).
\end{proof}

\begin{thm}
For $\mathcal{h},\mathcal{k},\mathcal{z}_i(i\in \mathbb{Z})\in\mathbb{C}$.\\
{\rm (1)} If $\mathcal{k}=0$ and there exists $r\in\mathbb{Z}$ such that $\mathcal{z}_r\ne0,\mathcal{z}_{r+i}=0 (\mbox{or}\,\, \mathcal{z}_{r-i}=0)\,(i\in\mathbb{Z}_+)$, the irreducible quotient of $M_{IV}(\mathcal{h},\mathcal{k},(\mathcal{z}))$ is isomorphic to the $\widehat{\mathfrak{age}}(1)$-module $U_q$ with generator $w$ and relations $\alpha_iw=0$ $(\alpha=e,p$ and $i\in\mathbb{Z})$, $h_iw=0$ $(i\in\mathbb{Z}\setminus \{0\})$, $hw=\mathcal{h}w, kw=\mathcal{k}w, z_iw=\mathcal{z}_iw (i\in \mathbb{Z})$.\\
{\rm (2)} If $\mathcal{k}\ne0$ and $(\mathcal{z})=(0)$, the irreducible quotient of $M_{IV}(\mathcal{h},\mathcal{k},(\mathcal{z}))$ is isomorphic to the $\widehat{\mathfrak{age}}(1)$-module $U_h$ with generator $w$ and relations $\alpha_iw=0$ $(\alpha=e,p,q,z$ and $i\in\mathbb{Z})$, $h_iw=0$ $(i\in\mathbb{Z}_+)$, $hw=\mathcal{h}w, kw=\mathcal{k}w$.\\
{\rm (3)} If $\mathcal{k}=0$ and $(\mathcal{z})=(0)$, the irreducible quotient of $M_{IV}(\mathcal{h},\mathcal{k},(\mathcal{z}))$ is isomorphic to the one-dimensional $\widehat{\mathfrak{age}}(1)$-module $\mathbb{C}w$.\\

\end{thm}
\begin{proof}
(1) By Lemma \ref{ivmlem} and $\mathcal{k}=0$,
for any ${\bf i}\in \mathbb{M}$, we have $\widehat{\mathfrak{age}}(1)^{[+]}h^{\bf i}v=0$.
Hence for any ${\bf 0\ne i}\in \mathbb{M}$, the submodule of $M_{IV}(\mathcal{h},\mathcal{k},(\mathcal{z}))$ generated by $h^{\bf i}v$ is proper.
Let $V^{\prime}$ be the submodule of $M_{IV}(\mathcal{h},\mathcal{k},(\mathcal{z}))$ generated by all $h^{\bf i}v\, ({\bf 0\ne i}\in \mathbb{M})$.

We consider the quotient module $U=M_{IV}(\mathcal{h},\mathcal{k},(\mathcal{z}))/V^{\prime}$,
it is similar to the Claim 1 of Theorem \ref{ivmthm},  we can get $U$ is simple.
Define
\begin{equation*}
\begin{aligned}
\varphi:&U\rightarrow U_q\\
&\sum_{{\bf i}\in \mathbb{M}}a_{\bf i,0}q_d^{\bf i}v+V^{\prime}\longmapsto \sum_{{\bf i}\in \mathbb{M}}a_{\bf i,0}q_d^{\bf i}w,
\end{aligned}
\end{equation*}
then $\varphi$ is a $\widehat{\mathfrak{age}}(1)$-module epimorphism, and $\varphi$ is injective since $U$ is simple, hence it is an isomorphism.

\vspace{3mm}
(2) By the PBW Theorem,
$$I=\{\cdots q_{-1}^{i_{-1}}q_{0}^{i_{0}}q_{1}^{i_{1}}\cdots\cdots h_{-2}^{j_{2}}h_{-1}^{j_{1}}v|i_r,j_s\in\mathbb{N},r\in\mathbb{Z},s\in\mathbb{Z}_+\}$$
forms a basis of $M_{IV}(\mathcal{h},\mathcal{k},(\mathcal{z}))$.

For any $v^\prime\in I$, denote $d(v^\prime)=\sum_{r\in \mathbb{Z}}{i_r}$ be the degree of $v^\prime\in I$.
For $v^{\prime\prime} \in M_{IV}(\mathcal{h},\mathcal{k},(\mathcal{z}))$, if $v^{\prime\prime}$ is a linear combination of the vectors in $I$ which have same degree $n\in\mathbb{N}$,
we denote $d(v^{\prime\prime})=n$.
Then $$d(q_sv^\prime)=d(v^\prime)+1,d(z_sv^\prime)=d(h_sv^\prime)=d(kv^\prime)=d(v^\prime)$$ for all $s\in\mathbb{Z}$.
By Lemma \ref{ivmlem} and $(\mathcal{z})=(0)$, we have $e_sv^\prime=p_sv^\prime=0$ for all $s\in\mathbb{Z}$.

Let $V^{\prime\prime}$ be the submodule of $M_{IV}(\mathcal{h},\mathcal{k},(\mathcal{z}))$ generated by all $v^\prime\in I$ with $d(v^\prime)>0$.
Then $V^{\prime\prime}$ doesn't contain $v$, hence is proper.
We consider the quotient module $U^\prime=M_{IV}(\mathcal{h},\mathcal{k},(\mathcal{z}))/V^{\prime\prime}$,
it is similar to the Claim 2 of Theorem \ref{ivmthm},  we can get $U^\prime$ is simple.
Define
\begin{equation*}
\begin{aligned}
\varphi^\prime:&U^\prime\rightarrow U_h\\
&\sum_{{\bf i}\in \mathbb{M}}a_{\bf i}h^{\bf i}v+V^{\prime\prime}\longmapsto \sum_{{\bf i}\in \mathbb{M}}a_{\bf i}h^{\bf i}w,
\end{aligned}
\end{equation*}
then $\varphi^\prime$ is a $\widehat{\mathfrak{age}}(1)$-module epimorphism, and $\varphi^\prime$ is injective since $U^\prime$ is simple, hence it is an isomorphism.

\vspace{3mm}
{\rm (3)} From (1) and (2), since $\mathcal{k}=0$ and $(\mathcal{z})=(0)$, then the simple quotient of $M_{IV}(\mathcal{h},\mathcal{k},(\mathcal{z}))$ is isomorphic to the one-dimensional $\widehat{\mathfrak{age}}(1)$-module $\mathbb{C}w$.
\end{proof}

\section{Simple restricted modules of $\widehat{\mathfrak{age}}(1)$}
\label{sec:4}
	\def\theequation{4.\arabic{equation}}
	\setcounter{equation}{0}

In this section, we present certain construction
of simple restricted $\widehat{\mathfrak{age}}(1)$-modules, then we give an explicit example
of this kind of simple restricted $\widehat{\mathfrak{age}}(1)$-module via the standard Whittaker module of $\widehat{\mathfrak{age}}(1)$.
First we give the definition of restricted modules of $\widehat{\mathfrak{age}}(1)$.
\begin{defi}
{\em An $\widehat{\mathfrak{age}}(1)$-module $W$ is called {\em restricted} if for any $w\in W$,
$e_{m}w=p_{m}w=q_{m}w=h_{m}w=z_{m}w=0$ for $m$ sufficiently large.}
\end{defi}

\subsection{Construction of simple restricted $\widehat{\mathfrak{age}}(1)$-modules}

Let $d_1,d_2,d_3\in \mathbb{Z}$, consider the subalgebra
\begin{equation}
\widehat{\mathfrak{age}}(1)_{(d_1,d_2,d_3)}=\sum_{i\in\mathbb{N}}(\mathbb{C}h_i\oplus\mathbb{C}e_{d_1+i}\oplus\mathbb{C}p_{d_2+i}\oplus\mathbb{C}q_{d_3+i})\oplus\sum_{i\in\mathbb{Z}}\mathbb{C}z_i\oplus\mathbb{C}k,
\end{equation}
then we have $d_1+d_3\geq d_2$.

For any simple $\widehat{\mathfrak{age}}(1)_{(d_1,d_2,d_3)}$-module $V$, we have the induced $\widehat{\mathfrak{age}}(1)$-module
$$\mbox{Ind}_{\widehat{\mathfrak{age}}(1)_{(d_1,d_2,d_3)}}^{\widehat{\mathfrak{age}}(1)}(V)=U(\widehat{\mathfrak{age}}(1))\otimes_{U(\widehat{\mathfrak{age}}(1)_{(d_1,d_2,d_3)})}V.$$

\begin{thm}\label{thm1}
Let $d_1,d_2,d_3\in \mathbb{Z}$ with $d_1+d_3\geq d_2$ and $V$ be a simple $\widehat{\mathfrak{age}}(1)_{(d_1,d_2,d_3)}$-module.
If there exist $l\in \mathbb{Z}_{> d_2}$ such that
\\{\rm (a)} $p_l$ acts injectively on $V$,
\\{\rm (b)} $p_{l+i}V=z_{l+d_2-d_1+i-1}V=e_{l-d_3+i}V=q_{l-d_1+i}V=h_{l-d_2+i}V=0$ for all $i\in \mathbb{Z}_+$.
\\Then $\mbox{Ind}_{\widehat{\mathfrak{age}}(1)_{(d_1,d_2,d_3)}}^{\widehat{\mathfrak{age}}(1)}(V)$ is a simple restricted $\widehat{\mathfrak{age}}(1)$-module.
\end{thm}

Before giving the proof of Theorem \ref{thm1}, we introduce some notions and notations that will be used later.
For ${\bf i,j,m,n}\in \mathbb{M}$, denote
$$p^{\bf i}e^{\bf j}q^{\bf m}h^{\bf n}=\cdots p_{-2+d_2}^{i_2}p_{-1+d_2}^{i_1}\cdots e_{-2+d_1}^{j_2}e_{-1+d_1}^{j_1}\cdots q_{-2+d_3}^{m_2}q_{-1+d_3}^{m_1}\cdots h_{-2}^{n_2}h_{-1}^{n_1}\in U(\widehat{\mathfrak{age}}(1)).$$
By the PBW Theorem, each element of $\mbox{Ind}_{\widehat{\mathfrak{age}}(1)_{(d_1,d_2,d_3)}}^{\widehat{\mathfrak{age}}(1)}(V)$ can be uniquely written in the form
\begin{equation}\label{eq:idelt}
  \sum_{{\bf i,j,m,n}\in \mathbb{M}}p^{\bf i}e^{\bf j}q^{\bf m}h^{\bf n}v_{\bf i,j,m,n},
\end{equation}
where all $v_{\bf i,j,m,n}\in V$ and only finitely many of them are nonzero.

\begin{defi}
{\em Define a {\em principal total order} on $\mathbb{M}\times\mathbb{M}\times\mathbb{M}\times\mathbb{M}$,
 still denoted by $\succ$: $\forall\ {\bf i,j,m,n,i^\prime,j^\prime,m^\prime,n^\prime} \in \mathbb{M} $,
 we say that $({\bf i,j,m,n}) \succ ({\bf i^\prime,j^\prime,m^\prime,n^\prime})$ if one of the following conditions holds.
 \\(i) ${\bf w(m+n)>w(m^\prime+n^\prime)}$.
 \\(ii) ${\bf w(m+n)=w(m^\prime+n^\prime),n\succ n^\prime}$.
 \\(iii) ${\bf w(m+n)=w(m^\prime+n^\prime),n=n^\prime,m\succ m^\prime}$.
 \\(iv) ${\bf n=n^\prime,m=m^\prime,w(i+j)>w(i^\prime+j^\prime)}$.
 \\(v) ${\bf n=n^\prime,m=m^\prime,w(i+j)=w(i^\prime+j^\prime),j\succ j^\prime}$.
 \\(vi) ${\bf n=n^\prime,m=m^\prime,w(i+j)=w(i^\prime+j^\prime),j=j^\prime,i\succ i^\prime}$.}
\end{defi}

For any $v\in\mbox{Ind}_{\widehat{\mathfrak{age}}(1)_{(d_1,d_2,d_3)}}^{\widehat{\mathfrak{age}}(1)}(V)$ written in the form of (\ref{eq:idelt}),
 we denote by $\mbox{supp}(v)$ the set of all $({\bf i,j,m,n)}\in \mathbb{M}\times\mathbb{M}\times\mathbb{M}\times\mathbb{M}$
 such that $v_{\bf i,j,m,n}\ne 0$.
 For a nonzero $v\in \mbox{Ind}_{\widehat{\mathfrak{age}}(1)_{(d_1,d_2,d_3)}}^{\widehat{\mathfrak{age}}(1)}(V)$,
  let $\mbox{deg}(v)$ denote the maximal
   (with respect to the principal total order on $\mathbb{M}\times\mathbb{M}\times\mathbb{M}\times\mathbb{M}$)
   element in $\mbox{supp}(v)$, called the {\em degree} of $v$.

\begin{lem}\label{lem1}
Let $d_1,d_2,d_3\in \mathbb{Z}$ with $d_1+d_3\geq d_2$ and $V$ be a simple $\widehat{\mathfrak{age}}(1)_{(d_1,d_2,d_3)}$-module
satisfying the conditions in Theorem \ref{thm1}.
For any $v\in \mbox{Ind}_{\widehat{\mathfrak{age}}(1)_{(d_1,d_2,d_3)}}^{\widehat{\mathfrak{age}}(1)}(V)\backslash V$
  written in the form of (\ref{eq:idelt}),
  denote ${\rm deg}(v)= ({\bf i,j,m,n})$.
  \vspace{3mm}\\
 {\rm (1)} If ${\bf n\ne 0}$, let $r={\rm min}\{s:n_s \ne 0\}>0$,
 then ${\rm deg}(p_{l+r}v) =({\bf i,j,m},{\bf n}-\epsilon_{r})$.\\
 {\rm (2)} If ${\bf n=0, m\ne0}$, let $a={\rm min}\{s:m_s \ne 0\}>0$,
 then ${\rm deg}(e_{a+l-d_3}v) =({\bf i,j},{\bf m}-\epsilon_{a},{\bf 0})$.\\
 {\rm (3)} If ${\bf n=m=0, j\ne 0}$, let $b={\rm min}\{j:j_s \ne 0\}>0$,
 then ${\rm deg}(q_{b-d_1+l}v) =({\bf i},{\bf j}-\epsilon_{b},{\bf 0,0})$.\\
 {\rm (4)} If ${\bf n=m=j=0, i\ne 0}$, let $c={\rm min}\{s:i_s \ne 0\}>0$,
 then ${\rm deg}(h_{c-d_2+l}v) =({\bf i}-\epsilon_{c},{\bf 0,0,0})$.\\
\end{lem}
\begin{proof}
(1) It suffices to consider those $v_{\bf i^\prime,j^\prime,m^\prime,n^\prime}$ with $p_{l+r}p^{\bf i^\prime}e^{\bf j^\prime}q^{\bf m^\prime}h^{\bf n^\prime}v_{\bf i^\prime,j^\prime,m^\prime,n^\prime}\ne0$.
Noticing that $p_{l+r}v_{\bf i^\prime,j^\prime,m^\prime,n^\prime}=0$ for any $({\bf i^\prime,j^\prime,m^\prime,n^\prime})\in \mbox{supp}(v)$.
We have
$$p_{l+r}p^{\bf i^\prime}e^{\bf j^\prime}q^{\bf m^\prime}h^{\bf n^\prime}v_{\bf i^\prime,j^\prime,m^\prime,n^\prime}=
p^{\bf i^\prime}e^{\bf j^\prime}[p_{l+r},q^{\bf m^\prime}]h^{\bf n^\prime}v_{\bf i^\prime,j^\prime,m^\prime,n^\prime}+
p^{\bf i^\prime}e^{\bf j^\prime}q^{\bf m^\prime}[p_{l+r},h^{\bf n^\prime}]v_{\bf i^\prime,j^\prime,m^\prime,n^\prime}.$$

For the first term $p^{\bf i^\prime}e^{\bf j^\prime}[p_{l+r},q^{\bf m^\prime}]h^{\bf n^\prime}v_{\bf i^\prime,j^\prime,m^\prime,n^\prime}$.
Consider
$$p^{\bf i^\prime}e^{\bf j^\prime}q^{\bf m^*}[p_{l+r},q_{-y+d_3}]q^{{\bf m^\prime-m^*}-\epsilon_y}h^{\bf n^\prime}v_{\bf i^\prime,j^\prime,m^\prime,n^\prime}=
p^{\bf i^\prime}e^{\bf j^\prime}q^{{\bf m^\prime}-\epsilon_y}h^{\bf n^\prime}z_{l+d_3+r-y}v_{\bf i^\prime,j^\prime,m^\prime,n^\prime},$$
where ${\bf m^*}\in\mathbb{M}$ such that $m_s^*=m^\prime_s$ for $s>y$ and $m_s^*=0$ for $s\leq y$,
since $d_3\geq d_2-d_1$, then this is zero for $y\leq r$.
If it is nonzero, denote
$$\mbox{deg}(p^{\bf i^\prime}e^{\bf j^\prime}q^{{\bf m^\prime}-\epsilon_y}h^{\bf n^\prime}z_{l+d_3+r-y}v_{\bf i^\prime,j^\prime,m^\prime,n^\prime})
=({\bf i_y,j_y,m_y,n_y}).$$
Then $({\bf i_y,j_y,m_y,n_y})=({\bf i^\prime,j^\prime,m^\prime}-\epsilon_y,{\bf n^\prime})$,
$${\bf w(m_y+n_y)=w(m^\prime+n^\prime)}-y<{\bf w(m^\prime+n^\prime)}-r\leq {\bf w(m+n)}-r={\bf w(m+n}-\epsilon_r),$$
hence $({\bf i_y,j_y,m_y,n_y})\prec ({\bf i,j,m},{\bf n}-\epsilon_{r})$.

For the second term $p^{\bf i^\prime}e^{\bf j^\prime}q^{\bf m^\prime}[p_{l+r},h^{\bf n^\prime}]v_{\bf i^\prime,j^\prime,m^\prime,n^\prime}$.
Consider
$$p^{\bf i^\prime}e^{\bf j^\prime}q^{\bf m^\prime}h^{\bf n^*}[p_{l+r},h_{-x}]h^{{\bf n^\prime-n^*}-\epsilon_x}v_{\bf i^\prime,j^\prime,m^\prime,n^\prime}=
-p^{\bf i^\prime}e^{\bf j^\prime}q^{\bf m^\prime}h^{\bf n^*}p_{l+r-x}h^{{\bf n^\prime-n^*}-\epsilon_x}v_{\bf i^\prime,j^\prime,m^\prime,n^\prime},$$
where ${\bf n^*}\in\mathbb{M}$ such that $n_s^*=n^\prime_s$ for $s>x$ and $n_s^*=0$ for $s\leq x$.
Denote
$$\mbox{deg}(p^{\bf i^\prime}e^{\bf j^\prime}q^{\bf m^\prime}h^{\bf n^*}p_{l+r-x}h^{{\bf n^\prime-n^*}-\epsilon_x}v_{\bf i^\prime,j^\prime,m^\prime,n^\prime})
=({\bf i_x,j_x,m_x,n_x}).$$
If $x>l+r-d_2$, we have $({\bf i_x,j_x,m_x,n_x})=({\bf i^\prime}+\epsilon_{x-l-r+d_2},{\bf j^\prime,m^\prime,n^\prime}-\epsilon_{x})$;
if $r<x\leq l+r-d_2$, we have $({\bf i_x,j_x,m_x,n_x})\preceq({\bf i^\prime},{\bf j^\prime,m^\prime,n^\prime}-\epsilon_{x})$,
then $${\bf w(m_x+n_x)}\leq{\bf w(m^\prime+n^\prime)}-x<{\bf w(m+n)}-r,$$
hence $({\bf i_x,j_x,m_x,n_x})\prec ({\bf i,j,m},{\bf n}-\epsilon_{r})$.
If $x<r$, we have ${\bf w(m^\prime+n^\prime)}<{\bf w(m+n)}$, otherwise it is a contradiction to
${\bf (i^\prime,j^\prime,m^\prime,n^\prime)}\preceq{\bf (i,j,m,n)}$, then
$${\bf w(m_x+n_x)}\leq{\bf w(m^\prime+n^\prime)}-r<{\bf w(m+n)}-r,$$
hence $({\bf i_x,j_x,m_x,n_x})\prec ({\bf i,j,m},{\bf n}-\epsilon_{r})$.
If $x=r$, we have $({\bf i_x,j_x,m_x,n_x})=({\bf i^\prime},{\bf j^\prime,m^\prime,n^\prime}-\epsilon_{r})\preceq ({\bf i,j,m},{\bf n}-\epsilon_{r})$,
the equality holds if and only if ${\bf i^\prime=i, j^\prime=j}$, ${\bf m^\prime=m, n^\prime=n}$.

Combining all the arguments, we see that ${\rm deg}(p_{l+r}v) =({\bf i,j,m},{\bf n}-\epsilon_{r})$.

\vspace{3mm}
(2) It suffices to consider those $v_{\bf i^\prime,j^\prime,m^\prime,n^\prime}$ with $e_{a+l-d_3}p^{\bf i^\prime}e^{\bf j^\prime}q^{\bf m^\prime}h^{\bf n^\prime}v_{\bf i^\prime,j^\prime,m^\prime,n^\prime}\ne0$.
Noticing that $e_{a+l-d_3}v_{\bf i^\prime,j^\prime,m^\prime,n^\prime}=0$ for any $({\bf i^\prime,j^\prime,m^\prime,n^\prime})\in \mbox{supp}(v)$.
We have
$$e_{a+l-d_3}p^{\bf i^\prime}e^{\bf j^\prime}q^{\bf m^\prime}h^{\bf n^\prime}v_{\bf i^\prime,j^\prime,m^\prime,n^\prime}=
p^{\bf i^\prime}e^{\bf j^\prime}[e_{a+l-d_3},q^{\bf m^\prime}]h^{\bf n^\prime}v_{\bf i^\prime,j^\prime,m^\prime,n^\prime}+
p^{\bf i^\prime}e^{\bf j^\prime}q^{\bf m^\prime}[e_{a+l-d_3},h^{\bf n^\prime}]v_{\bf i^\prime,j^\prime,m^\prime,n^\prime}.$$

For the second term $p^{\bf i^\prime}e^{\bf j^\prime}q^{\bf m^\prime}[e_{a+l-d_3},h^{\bf n^\prime}]v_{\bf i^\prime,j^\prime,m^\prime,n^\prime}$.
We have ${\bf w(m^\prime+n^\prime)}<{\bf w(m)}$, otherwise it is a contradiction to
${\bf (i^\prime,j^\prime,m^\prime,n^\prime)}\preceq{\bf (i,j,m,0)}$.
Consider
$$p^{\bf i^\prime}e^{\bf j^\prime}q^{\bf m^\prime}h^{\bf n^*}[e_{a+l-d_3},h_{-x}]h^{{\bf n^\prime-n^*}-\epsilon_x}v_{\bf i^\prime,j^\prime,m^\prime,n^\prime}
=-2p^{\bf i^\prime}e^{\bf j^\prime}q^{\bf m^\prime}h^{\bf n^*}e_{a-x+l-d_3}h^{{\bf n^\prime-n^*}-\epsilon_x}v_{\bf i^\prime,j^\prime,m^\prime,n^\prime},$$
where ${\bf n^*}\in\mathbb{M}$ such that $n_s^*=n^\prime_s$ for $s>x$ and $n_s^*=0$ for $s\leq x$.
Denote
$$\mbox{deg}(p^{\bf i^\prime}e^{\bf j^\prime}q^{\bf m^\prime}h^{\bf n^*}e_{a-x+l-d_3}h^{{\bf n^\prime-n^*}-\epsilon_x}v_{\bf i^\prime,j^\prime,m^\prime,n^\prime})
=({\bf i_x,j_x,m_x,n_x}),$$
then we have $${\bf w(m_x+n_x)}\leq{\bf w(m^\prime+n^\prime)}-a<{\bf w(m)}-a={\bf w(m}-\epsilon_a),$$
hence $({\bf i_x,j_x,m_x,n_x})\prec({\bf i,j},{\bf m}-\epsilon_{a},{\bf 0})$.

For the first term $p^{\bf i^\prime}e^{\bf j^\prime}[e_{a+l-d_3},q^{\bf m^\prime}]h^{\bf n^\prime}v_{\bf i^\prime,j^\prime,m^\prime,n^\prime}$.
Consider
$$p^{\bf i^\prime}e^{\bf j^\prime}q^{\bf m^*}[e_{a+l-d_3},q_{-y+d_3}]q^{{\bf m^\prime-m^*}-\epsilon_y}h^{\bf n^\prime}v_{\bf i^\prime,j^\prime,m^\prime,n^\prime}=
p^{\bf i^\prime}e^{\bf j^\prime}q^{\bf m^*}p_{a-y+l}q^{{\bf m^\prime-m^*}-\epsilon_y}h^{\bf n^\prime}v_{\bf i^\prime,j^\prime,m^\prime,n^\prime},$$
where ${\bf m^*}\in\mathbb{M}$ such that $m_s^*=m^\prime_s$ for $s>y$ and $m_s^*=0$ for $s\leq y$.
Denote
$$\mbox{deg}(p^{\bf i^\prime}e^{\bf j^\prime}q^{\bf m^*}p_{a-y+l}q^{{\bf m^\prime-m^*}-\epsilon_y}h^{\bf n^\prime}v_{\bf i^\prime,j^\prime,m^\prime,n^\prime})
=({\bf i_y,j_y,m_y,n_y}).$$
If $y>l+a-d_2$, we have $({\bf i_y,j_y,m_y,n_y})=({\bf i^\prime}+\epsilon_{y-l-a+d_2},{\bf j^\prime,m^\prime}-\epsilon_{y},{\bf n^\prime})$;
if $a<y\leq l+a-d_2$, we have $({\bf i_y,j_y,m_y,n_y})\preceq({\bf i^\prime},{\bf j^\prime,m^\prime}-\epsilon_{y},{\bf n^\prime})$,
then $${\bf w(m_y+n_y)}\leq{\bf w(m^\prime+n^\prime)}-y<{\bf w(m)}-a,$$
hence $({\bf i_y,j_y,m_y,n_y})\prec ({\bf i,j},{\bf m}-\epsilon_{a},{\bf 0})$.
If $y<a$, ${\bf w(m^\prime+n^\prime)}<{\bf w(m)}$, otherwise it is a contradiction to
${\bf (i^\prime,j^\prime,m^\prime,n^\prime)}\preceq{\bf (i,j,m,0)}$, we have
$${\bf w(m_y+n_y)}\leq{\bf w(m^\prime+n^\prime)}-a<{\bf w(m)}-a,$$
hence $({\bf i_y,j_y,m_y,n_y})\prec ({\bf i,j},{\bf m}-\epsilon_{a},{\bf 0})$.
If $y=a$, we have $({\bf i_y,j_y,m_y,n_y})=({\bf i^\prime},{\bf j^\prime,m^\prime}-\epsilon_{a},{\bf n^\prime})\preceq ({\bf i,j,m}-\epsilon_{a},{\bf 0})$,
the equality holds if and only if ${\bf i^\prime=i, j^\prime=j}$, ${\bf m^\prime=m, n^\prime=0}$.

Combining all the arguments, we see that ${\rm deg}(e_{a+l-d_3}v) =({\bf i,j},{\bf m}-\epsilon_{a},{\bf 0})$.

\vspace{3mm}
(3) It suffices to consider those $v_{\bf i^\prime,j^\prime,0,0}$ with $q_{b-d_1+l}p^{\bf i^\prime}e^{\bf j^\prime}v_{\bf i^\prime,j^\prime,0,0}\ne0$.

Noticing that $q_{b-d_1+l}v_{\bf i^\prime,j^\prime,0,0}=0$ for any $({\bf i^\prime,j^\prime,0,0})\in \mbox{supp}(v)$.
We have
$$q_{b-d_1+l}p^{\bf i^\prime}e^{\bf j^\prime}v_{\bf i^\prime,j^\prime,0,0}=
[q_{b-d_1+l},p^{\bf i^\prime}]e^{\bf j^\prime}v_{\bf i^\prime,j^\prime,0,0}+
p^{\bf i^\prime}[q_{b-d_1+l},e^{\bf j^\prime}]v_{\bf i^\prime,j^\prime,0,0}.$$

For the first term $[q_{b-d_1+l},p^{\bf i^\prime}]e^{\bf j^\prime}v_{\bf i^\prime,j^\prime,0,0}$.
Consider
$$p^{\bf i^*}[q_{b-d_1+l},p_{-y+d_2}]p^{{\bf i^\prime-i^*}-\epsilon_y}e^{\bf j^\prime}v_{\bf i^\prime,j^\prime,0,0}=
-p^{{\bf i^\prime}-\epsilon_y}e^{\bf j^\prime}z_{b-y+d_2-d_1+l}v_{\bf i^\prime,j^\prime,0,0},$$
where ${\bf i^*}\in\mathbb{M}$ such that $i_s^*=i^\prime_s$ for $s>y$ and $i_s^*=0$ for $s\leq y$,
then this is zero for $y\leq b$. If it is nonzero, denote
$$\mbox{deg}(p^{{\bf i^\prime}-\epsilon_y}e^{\bf j^\prime}z_{b-y+d_2-d_1+l}v_{\bf i^\prime,j^\prime,0,0})
=({\bf i_y,j_y,0,0}).$$
Then $({\bf i_y,j_y,0,0})=({\bf i^\prime}-\epsilon_y,{\bf j^\prime},{\bf 0,0})$,
$${\bf w(i_y+j_y)=w(i^\prime+j^\prime)}-y<{\bf w(i^\prime+j^\prime)}-b\leq {\bf w(i+j)}-b={\bf w(i+j}-\epsilon_b),$$
hence $({\bf i_y,j_y,0,0})\prec ({\bf i,j}-\epsilon_{b},{\bf 0,0})$.

For the second term $p^{\bf i^\prime}[q_{b-d_1+l},e^{\bf j^\prime}]v_{\bf i^\prime,j^\prime,0,0}$.
Consider
$$p^{\bf i^\prime}e^{\bf j^*}[q_{b-d_1+l},e_{-x+d_1}]e^{{\bf j^\prime-j^*}-\epsilon_x}v_{\bf i^\prime,j^\prime,0,0}=
-p^{\bf i^\prime}e^{{\bf j^\prime}-\epsilon_x}p_{b-x+l}v_{\bf i^\prime,j^\prime,0,0},$$
where ${\bf j^*}\in\mathbb{M}$ such that $j_s^*=j^\prime_s$ for $s>x$ and $j_s^*=0$ for $s\leq x$,
then this is zero for $x<b$. If it is nonzero, denote
$$\mbox{deg}(p^{\bf i^\prime}e^{{\bf j^\prime}-\epsilon_x}p_{b-x+l}v_{\bf i^\prime,j^\prime,0,0})
=({\bf i_x,j_x,0,0}).$$
If $x>l+b-d_2$, we have $({\bf i_x,j_x,0,0})=({\bf i^\prime}+\epsilon_{x-l-b+d_2},{\bf j^\prime}-\epsilon_{x},{\bf 0,0})$,
since $l>d_2$, $${\bf w(i_x+j_x)}={\bf w(i^\prime+j^\prime)}-x+x-l-b+d_2<{\bf w(i+j)}-b,$$
hence $({\bf i_x,j_x,0,0})\prec ({\bf i,j}-\epsilon_{b},{\bf 0,0})$.
If $b<x\leq l+b-d_2$, we have $({\bf i_x,j_x,0,0})\preceq({\bf i^\prime},{\bf j^\prime}-\epsilon_{x},{\bf 0,0})$,
then ${\bf w(i_x+j_x)}\leq{\bf w(i^\prime+j^\prime)}-x<{\bf w(i+j)}-b,$
hence $({\bf i_x,j_x,0,0})\prec ({\bf i,j}-\epsilon_{b},{\bf 0,0})$.
If $x=b$, we have $({\bf i_x,j_x,0,0})=({\bf i^\prime},{\bf j^\prime}-\epsilon_{b},{\bf 0,0})\preceq({\bf i,j}-\epsilon_{b},{\bf 0,0})$,
the equality holds if and only if ${\bf i^\prime=i, j^\prime=j}$.

Combining all the arguments, we see that ${\rm deg}(q_{b-d_1+l}v) =({\bf i},{\bf j}-\epsilon_{b},{\bf 0,0})$.

\vspace{3mm}
(4) It suffices to consider those $v_{\bf i^\prime,j^\prime,0,0}$ with $h_{c-d_2+l}p^{\bf i^\prime}e^{\bf j^\prime}v_{\bf i^\prime,j^\prime,0,0}\ne0$.

Noticing that $h_{c-d_2+l}v_{\bf i^\prime,j^\prime,0,0}=0$ for any $({\bf i^\prime,j^\prime,0,0})\in \mbox{supp}(v)$.
We have
$$h_{c-d_2+l}p^{\bf i^\prime}e^{\bf j^\prime}v_{\bf i^\prime,j^\prime,0,0}=
[h_{c-d_2+l},p^{\bf i^\prime}]e^{\bf j^\prime}v_{\bf i^\prime,j^\prime,0,0}+
p^{\bf i^\prime}[h_{c-d_2+l},e^{\bf j^\prime}]v_{\bf i^\prime,j^\prime,0,0}.$$

For the second term $p^{\bf i^\prime}[h_{c-d_2+l},e^{\bf j^\prime}]v_{\bf i^\prime,j^\prime,0,0}$.
We have ${\bf w(i^\prime+j^\prime)}<{\bf w(i)}$, otherwise it is a contradiction to
${\bf (i^\prime,j^\prime,0,0)}\preceq{\bf (i,0,0,0)}$.
Consider
$$p^{\bf i^\prime}e^{\bf j^*}[h_{c-d_2+l},e_{-y+d_1}]e^{{\bf j^\prime-j^*}-\epsilon_y}v_{\bf i^\prime,j^\prime,0,0}
=2p^{\bf i^\prime}e^{{\bf j^\prime}-\epsilon_y}e_{c-y+l+d_1-d_2}v_{\bf i^\prime,j^\prime,0,0},$$
where ${\bf j^*}\in\mathbb{M}$ such that $j_s^*=j^\prime_s$ for $s>y$ and $j_s^*=0$ for $s\leq y$.
Since $d_1-d_2\geq-d_3$, then this is zero for $y<c$, if it is nonzero,
denote
$$\mbox{deg}(p^{\bf i^\prime}e^{{\bf j^\prime}-\epsilon_y}e_{c-y+l+d_1-d_2}v_{\bf i^\prime,j^\prime,0,0})
=({\bf i_y,j_y,0,0}),$$
then we have $${\bf w(i_y+j_y)}\leq{\bf w(i^\prime+j^\prime)}-c<{\bf w(i)}-c={\bf w(i}-\epsilon_c),$$
hence $({\bf i_y,j_y,0,0})\prec({\bf i}-\epsilon_{c},{\bf 0,0,0})$.

For the first term $[h_{c-d_2+l},p^{\bf i^\prime}]e^{\bf j^\prime}v_{\bf i^\prime,j^\prime,0,0}$.
Consider
$$p^{\bf i^*}[h_{c-d_2+l},p_{-x+d_2}]p^{{\bf i^\prime-i^*}-\epsilon_x}e^{\bf j^\prime}v_{\bf i^\prime,j^\prime,0,0}=
p^{{\bf i^\prime}-\epsilon_x}e^{\bf j^\prime}p_{c-x+l}v_{\bf i^\prime,j^\prime,0,0},$$
where ${\bf i^*}\in\mathbb{M}$ such that $i_s^*=i^\prime_s$ for $s>x$ and $i_s^*=0$ for $s\leq x$.
Then this is zero for $x<c$, if it is nonzero,
denote
$\mbox{deg}(p^{{\bf i^\prime}-\epsilon_x}e^{\bf j^\prime}p_{c-x+l}v_{\bf i^\prime,j^\prime,0,0})
=({\bf i_x,j_x,0,0}).$
If $x>l+c-d_2$, we have $({\bf i_x,j_x,0,0})=({\bf i^\prime}+\epsilon_{x-l-c+d_2}-\epsilon_{x},{\bf j^\prime,0},{\bf 0})$,
then $${\bf w(i_x+j_x)}={\bf w(i^\prime+j^\prime)}-x+x-l-c+d_2<{\bf w(i)}-c,$$
hence $({\bf i_x,j_x,0,0})\prec ({\bf i}-\epsilon_{c},{\bf 0,0,0})$.
If $c<x\leq l+c-d_2$, we have $({\bf i_x,j_x,0,0})\preceq({\bf i^\prime}-\epsilon_{x},{\bf j^\prime},{\bf 0,0})$,
then ${\bf w(i_x+j_x)}\leq{\bf w(i^\prime+j^\prime)}-x<{\bf w(i)}-c,$
hence $({\bf i_x,j_x,0,0})\prec ({\bf i}-\epsilon_{c},{\bf 0,0,0})$.
If $x=c$, we have $({\bf i_x,j_x,0,0})=({\bf i^\prime}-\epsilon_{c},{\bf j^\prime},{\bf 0,0})\preceq({\bf i}-\epsilon_{c},{\bf 0,0,0})$,
the equality holds if and only if ${\bf i^\prime=i, j^\prime=0}$.

Combining all the arguments, we see that ${\rm deg}(h_{c-d_2+l}v) =({\bf i}-\epsilon_{c},{\bf 0,0,0})$.
\end{proof}

\vspace{3mm}
\noindent{\em Proof of Theorem \ref{thm1}.}
For any nonzero element $v\in \mbox{Ind}_{\widehat{\mathfrak{age}}(1)_{(d_1,d_2,d_3)}}^{\widehat{\mathfrak{age}}(1)}(V)$,
by using Lemma \ref{lem1} repeatedly,
we can obtain a nonzero element in
$U(\widehat{\mathfrak{age}}(1))v\cap V$,
which gives the simplicity of $\mbox{Ind}_{\widehat{\mathfrak{age}}(1)_{(d_1,d_2,d_3)}}^{\widehat{\mathfrak{age}}(1)}(V)$.
 For any $v\in \mbox{Ind}_{\widehat{\mathfrak{age}}(1)_{(d_1,d_2,d_3)}}^{\widehat{\mathfrak{age}}(1)}(V)$,
  we can take
  $$m>{\bf w(i^\prime+j^\prime+m^\prime+n^\prime)}+l+|d_1|{\bf d(j^\prime)}+|d_2|{\bf d(i^\prime)}+|d_3|{\bf d(m^\prime)}$$
  for all ${\bf (i^\prime,j^\prime,m^\prime,n^\prime)}\in\mbox{supp}(v)$,
  then by direct computations we get $$e_{m}v=p_{m}v=q_{m}v=h_{m}v=z_{m}v=0.$$
  Hence $\mbox{Ind}_{\widehat{\mathfrak{age}}(1)_{(d_1,d_2,d_3)}}^{\widehat{\mathfrak{age}}(1)}(V)$ is simple restricted $\widehat{\mathfrak{age}}(1)$-module.
\,\,\,\,\,\,\,\,\,\,\,\,\,\,\,\,\,\,\,\,\,\,\,\,\,\,\,\,\,\,\,\,\,\,\,
\,\,\,\,\,\,\,\,\,\,\,\,\,\,\,\,\,\,\,\,\,\,\,\,\,\,\,\,\,\,\,\,\,\,\,\,\,\,\,
$\Box$

\vspace{3mm}
Let $S$ be a simple restricted $\widehat{\mathfrak{age}}(1)$-module such that there exist $a\in \mathbb{Z}$ such that the action of $p_a$ on $S$ are injective.
Since $S$ is simple, then $z_m$ acts on $S$ as scalar $\mathcal{z}_m$ for all $m\in\mathbb{Z}$.
For any $x_1,x_2,x_3,x_4\in\mathbb{Z}$,
 consider the vector space
\begin{equation}\nonumber
N_{x_1,x_2,x_3,x_4}=\{ v\in S\ | \ p_{x_1+i}v=e_{x_2+i}v=h_{x_3+i}v=q_{x_4+i}v=0 \,\,\mbox{for all}\,\,i\in\mathbb{Z}_+\}.
\end{equation}
Since $S$ is a restricted $\widehat{\mathfrak{age}}(1)$-module,
we know that $N_{x_1,x_2,x_3,x_4}\ne 0$ for sufficiently large integers $x_1,x_2,x_3,x_4$
and there exist $r\in \mathbb{Z}$ such that $\mathcal{z}_{r-1}\ne0, \mathcal{z}_{r+i}=0$ for all $i\in \mathbb{N}$.
Note that if $N_{x_1,x_2,x_3,x_4}\ne 0$, then for any $i_{1},i_{2},i_{3},i_4\in\mathbb{N}$,
$N_{x_1+i_{1},x_2+i_{2},x_3+i_{3},x_4+i_{4}}\neq 0$.
Since $p_{a}$ acts injectively on $S$, we can find
smallest integer $l\in\mathbb{Z}_{\geqslant a}$ such
that $N_{l,x_2,x_3,x_4}\neq0$.
Let $l^\prime\in\mathbb{Z}$ be the smallest integer such
that $N_{l,l^\prime,x_3,x_4}\neq0$.

\begin{thm}
If $r\leq 2l-l^\prime$,
then there exists $d_1,d_2,d_3\in\mathbb{Z}$ and a simple $\widehat{\mathfrak{age}}(1)_{(d_1,d_2,d_3)}$-module $V$
satisfying the conditions in Theorem \ref{thm1} such that $S\cong {\rm Ind}_{\widehat{\mathfrak{age}}(1)_{(d_1,d_2,d_3)}}^{\widehat{\mathfrak{age}}(1)}(V)$.
\end{thm}

\begin{proof}
Since $r\leq 2l-l^\prime$,
then we can take $x_3,x_4$ such that $l-r\geq x_3-x_4\geq l^\prime-l$ and $x_3\geq1$.
Assume that $p_{l}$ does not act injectively on $N_{l,l^\prime,x_3,x_4}$.
Then $N_{l-1,l^\prime,x_3,x_4}\neq 0$ which contradicts our choice of $l$.
Thus $p_{l}$ acts injectively on $N_{l,l^\prime,x_3,x_4}$.
Let $d_1=l-x_4,d_2=l-x_3,d_3=l-l^\prime,V=N_{l,l^\prime,x_3,x_4}$,
then $d_1+d_3=l-x_4+l-l^\prime\geq l-x_3=d_2$, $l>l-x_3=d_2$,
$l+d_2-d_1=l-x_3+x_4\geq r$,
hence $p_l$ acts injectively on $V$ and $p_{l+i}V=z_{l+d_2-d_1+i-1}V=e_{l-d_3+i}V=q_{l-d_1+i}V=h_{l-d_2+i}V=0$ for all $i\in \mathbb{Z}_+$.
For all $v\in V,i\in\mathbb{Z}_+,j\in\mathbb{N}$, since $d_3-d_2\geq -d_1$,
 \begin{equation*}
\begin{aligned}
&h_{x_3+i}q_{d_3+j}v=[h_{x_3+i},q_{d_3+j}]v=-q_{l-d_2+d_3+i+j}v=0,\\
&p_{l+i}q_{d_3+j}v=[p_{l+i},q_{d_3+j}]v=z_{l+d_3+i+j}v=0,\\
&e_{l^\prime+i}q_{d_3+j}v=[e_{l^\prime+i},q_{d_3+j}]v=p_{l^\prime+d_3+i+j}v=0,\\
&q_{x_4+i}q_{d_3+j}v=q_{d_3+j}q_{x_4+i}v=0,
\end{aligned}
\end{equation*}
then $q_{d_3+j}v\in V$, similarly we have $h_{j}v,e_{d_1+j}v,p_{d_2+j}v\in V$,
hence $V$ is a $\widehat{\mathfrak{age}}(1)_{(d_1,d_2,d_3)}$-module satisfying the conditions in Theorem \ref{thm1} except the simplicity.

There exists a canonical $\widehat{\mathfrak{age}}(1)$-module epimorphism
\begin{equation*}
\pi :{\rm Ind}_{\widehat{\mathfrak{age}}(1)_{(d_1,d_2,d_3)}}^{\widehat{\mathfrak{age}}(1)}(V) \rightarrow S,\;\;\pi (1\otimes v) =v\;\;\mbox{for any}\; v\in V.
\end{equation*}
Let $K=\mbox{ker}(\pi)$ be the kernel of $\pi$.
 It is clear that $K\cap V=0$.
 Note that $K$ is an $\widehat{\mathfrak{age}}(1)$-submodule of ${\rm Ind}_{\widehat{\mathfrak{age}}(1)_{(d_1,d_2,d_3)}}^{\widehat{\mathfrak{age}}(1)}(V)$
    and hence is stable under the actions of $e_i,p_i,h_i$ and $q_{i}$ for all $i\in\mathbb{Z}$.
 If $K\ne 0$, for any nonzero vector $v\in K$
  such that $\mbox{deg}(v)=({\bf i,j,m,n})$,
   by the proof of Lemmas \ref{lem1} (the proof does not use the simplicity of $V$),
we can obtain a vector $u\in V$,
 which is a contradiction.
 Thus we have $K=0$, that is, $S\cong {\rm Ind}_{\widehat{\mathfrak{age}}(1)_{(d_1,d_2,d_3)}}^{\widehat{\mathfrak{age}}(1)}(V) $.
 By the property of induced modules, we know $V$ is a simple $\widehat{\mathfrak{age}}(1)_{(d_1,d_2,d_3)}$-module.
\end{proof}

\subsection{Examples}
Now we give an example
of simple restricted $\widehat{\mathfrak{age}}(1)$-module constructed from the standard Whittaker $\widehat{\mathfrak{age}}(1)$-module.
First, we consider the subalgebra
\begin{equation}
\widehat{\mathfrak{age}}(1)_{+}=\sum_{i\in\mathbb{N}}(\mathbb{C}h_{i+1}\oplus\mathbb{C}e_{i}\oplus\mathbb{C}p_{i}\oplus\mathbb{C}q_{1+i})\oplus\sum_{i\in\mathbb{Z}}\mathbb{C}z_i\oplus\mathbb{C}k.
\end{equation}
\begin{defi}
{\em For a Lie algebra $L$, let $L_0:=L$ and $L_i:=[L_{i-1},L]$ for $i\in\mathbb{Z}_+$.
We say that $L$ is {\em quasi-nilpotent} if $\bigcap_{i=0}^{\infty}L_i=0$.}
\end{defi}

\begin{defi}
{\em Let $L$ be a Lie algebra and $W$ be an $L$-module. We say that $L$ acts {\em locally nilpotently} on $W$ if for any $w\in W$
  there exists $n\in \mathbb{Z}_+$ such that $\alpha_1\alpha_2\cdots\alpha_n(w)=0$ for all $\alpha_1,\alpha_2,\cdots,\alpha_n\in L$.}
\end{defi}

It is clear that $\widehat{\mathfrak{age}}(1)_{+}$ is quasi-nilpotent and
the action of $\widehat{\mathfrak{age}}(1)_{+}$ on the adjoint $\widehat{\mathfrak{age}}(1)_{+}$-module $\widehat{\mathfrak{age}}(1)/\widehat{\mathfrak{age}}(1)_{+}$ is locally nilpotent.
Then $(\widehat{\mathfrak{age}}(1),\widehat{\mathfrak{age}}(1)_{+})$ is a Whittaker pair (see \cite{BM}), hence we can define the {\em standard Whittaker module of type $\varphi$ associated to $(\widehat{\mathfrak{age}}(1),\widehat{\mathfrak{age}}(1)_{+})$} (also known as universal Whittaker module) by
$$W_{(\varphi,\widehat{\mathfrak{age}}(1)_{+})}=U(\widehat{\mathfrak{age}}(1))\otimes_{U(\widehat{\mathfrak{age}}(1)_{+})}\mathbb{C}w,$$
where $\varphi:\widehat{\mathfrak{age}}(1)_{+}\rightarrow\mathbb{C}$ is a Lie algebra homomorphism and $\mathbb{C}w$ is a $\widehat{\mathfrak{age}}(1)_{+}$-module
with $\alpha.w=\varphi(\alpha)w$ for any $\alpha\in \widehat{\mathfrak{age}}(1)_{+}$.
Since $\varphi:\widehat{\mathfrak{age}}(1)_{+}\rightarrow\mathbb{C}$ is a Lie algebra homomorphism,
then $\varphi(e_i)=\varphi(p_i)=\varphi(q_{i+1})=\varphi(z_i)=0$ for all $i\in\mathbb{Z}_+$.

Let $d_1=-1,d_2=-1,d_3=1$, we consider the subalgebra
$$\widehat{\mathfrak{age}}(1)_{(-1,-1,1)}=\sum_{i\in\mathbb{N}}(\mathbb{C}h_i\oplus\mathbb{C}e_{-1+i}\oplus\mathbb{C}p_{-1+i}\oplus\mathbb{C}q_{1+i})\oplus\sum_{i\in\mathbb{Z}}\mathbb{C}z_i\oplus\mathbb{C}k$$
and $\widehat{\mathfrak{age}}(1)_{(-1,-1,1)}$-module
$M=U(\widehat{\mathfrak{age}}(1)_{(-1,-1,1)})\otimes_{U(\widehat{\mathfrak{age}}(1)_{+})}\mathbb{C}w.$
It is clear that $$W_{(\varphi,\widehat{\mathfrak{age}}(1)_{+})}\cong U(\widehat{\mathfrak{age}}(1))\otimes_{U(\widehat{\mathfrak{age}}(1)_{(-1,-1,1)})}M.$$

\begin{prop}\label{exam}
If $\varphi(p)\ne0$ and $\varphi(z)=\varphi(e)=\varphi(h_i)=0,i\in\mathbb{Z}_{\geq2}$. Then $\widehat{\mathfrak{age}}(1)_{(-1,-1,1)}$-module
$M$ is simple and satisfies the conditions of Theorem \ref{thm1}.
\end{prop}
\begin{proof}
By the PBW Theorem, for any $v\in M\setminus\mathbb{C}w$, we can write $v$ in the form of the following finite sum
\begin{equation}
\sum_{j_1,j_2,j_3\in\mathbb{N}}a_{(j_3,j_2,j_1)}p_{-1}^{j_3}e_{-1}^{j_2}h^{j_1}w,
\end{equation}
where $a_{(j_3,j_2,j_1)}\in\mathbb{C}$.
Let $\mbox{supp}(v)$ be the set of all $(j_3,j_2,j_1)\in\mathbb{N}^3$ with $a_{(j_3,j_2,j_1)}\ne0$
and $\mbox{deg}(v)=(i_3,i_2,i_1)$ be the maximal element of $\mbox{supp}(v)$ with respect to the reverse lexicographical order.
The following can be directly checked.
\begin{itemize}
\item[(1)] If $i_1\ne0$, since $\varphi(p)\ne0$, we have $\mbox{deg}((p-\varphi(p))v)=(i_3,i_2,i_1-1)$.
\item[(2)] If $i_1=0,i_2\ne0$, since $\varphi(p)\ne0$ and $\varphi(z)=0$, we have $\mbox{deg}((q_1-\varphi(q_1))v)=(i_3,i_2-1,0)$.
\item[(3)] If $i_1=i_2=0,i_3\ne0$, since $\varphi(p)\ne0$, we have $\mbox{deg}((h_1-\varphi(h_1))v)=(i_3-1,0,0)$.
\end{itemize}
These show that $M$ is a simple $\widehat{\mathfrak{age}}(1)_{(-1,-1,1)}$-module.
The remaining parts follow from direct computation.
\end{proof}

If $\varphi(p)\ne0$ and $\varphi(z)=\varphi(e)=\varphi(h_i)=0,i\in\mathbb{Z}_{\geq2}$,
by Proposition \ref{exam} and Theorem \ref{thm1}, $W_{(\varphi,\widehat{\mathfrak{age}}(1)_{+})}$ is a simple restricted $\widehat{\mathfrak{age}}(1)$-module.

\end{document}